\documentclass[twoside, 12pt]{article}
\usepackage{mathrsfs}
\usepackage{amssymb, amsmath, mathrsfs, amsthm}
\usepackage{graphicx}
\usepackage{color}
\usepackage{tikz}
\usepackage[top=2cm, bottom=2cm, left=2cm, right=2cm]{geometry}
\usepackage{float, caption, subcaption}
    \usepackage{diagbox}
\usepackage{algorithm}
\usepackage{algpseudocode}
\usepackage{amsmath}
\usepackage[colorlinks,
  linkcolor=blue,
 anchorcolor=blue,
 citecolor=blue]{hyperref} 
\input{epsf}

\newcommand{\bd}{\begin{description}}
\newcommand{\ed}{\end{description}}
\newcommand{\bi}{\begin{itemize}}
\newcommand{\ei}{\end{itemize}}
\newcommand{\be}{\begin{enumerate}}
\newcommand{\ee}{\end{enumerate}}
\newcommand{\beq}{\begin{equation}}
\newcommand{\eeq}{\end{equation}}
\newcommand{\beqs}{\begin{eqnarray*}}
\newcommand{\eeqs}{\end{eqnarray*}}

\definecolor{DarkGreen}{rgb}{0.2, 0.6, 0.3}

\usepackage{multicol}
\usepackage{amsthm}
\usepackage{multirow}
\usepackage{makecell}

\newtheorem{theorem}{Theorem}[section]

\newtheorem{assumption}{Assumption}[section]
\newtheorem{lemma}{Lemma}[section]
\newtheorem{definition}{Definition}[section]
\newtheorem{corollary}[theorem]{Corollary}
\newtheorem{case}{Case}

\newtheorem{remark}{Remark}[section]

\newtheorem{proposition}{Proposition}[section]

\begin{document}
\title{\textbf{Ramsey-Turán theory for partially-ordered sets
} \footnote{
The work was supported by the Hungarian National
Research, Development and Innovation Office
NKFIH (No.~SSN135643 and K132696); the National Science Foundation of China
(Nos. 12471329 and 12061059)}}

\author{Gyula O.H. Katona\thanks{HUN-REN Alfr\'ed R\'enyi Institute of Mathematics, Budapest Re\'altanoda utca 13-15, 1053, Hungary. {\tt
katona.gyula.oh@renyi.hu}}, \
\ Yaping Mao\footnote{Corresponding author: Academy of Plateau
Science and Sustainability, and School of Mathematics and Statistics, Qinghai Normal University, Xining, Qinghai 810008, China.
{\tt yapingmao@outlook.com; myp@qhnu.edu.cn}}}
\date{}
\maketitle

\begin{abstract}
We say that a poset $Q$
contains a copy (resp.~an induced copy) of a poset $P$
if there is an injection $f : P \to Q$
such that for any $x,y \in P$,
$f(x)\leq f(y)$ in $Q$ if (resp.~if and only if) $x\leq y$ in $P$.
Let $\mathcal{Q} = \{Q_n\}_{n\in\mathbb{N}}$ be a host poset family, $P$ a poset, and $n,l,t\in\mathbb{N}$. The \emph{weak poset Ramsey-Turán number for $t$-chains}, denoted $\operatorname{RT}(\mathcal{Q}; n, P, l, t)$, resp. the \emph{strong poset Ramsey-Turán number for $t$-chains}, denoted $\operatorname{RT}^{\sharp}(\mathcal{Q}; n, P, l, t)$, is the maximum value of $|\mathcal{C}_t(Q)|$ over all posets $Q$ satisfying the following three conditions: $Q \subseteq Q_n$;
$Q$ is weakly $P$-free resp. strongly $P$-free;
$w(Q) < l$, where $w(Q)$ is the cardinality of a largest antichain in $Q$.
In this paper, we study the upper and lower bounds of poset Ramsey-Turán numbers, introducing weak and strong Boolean Ramsey-Turán numbers for \(t\)-chains (\(\operatorname{RT}(\mathcal{B}; n, P, l, t)\) and \(\operatorname{RT}^\sharp(\mathcal{B}; n, P, l, t)\)), where $\mathcal{B}=\{B_{n} : n\geq 1\}$ is
the family of Boolean lattices. We first prove that \(\operatorname{RT}(\mathcal{B};n,P,l,t) \leq \operatorname{RT}^\sharp(\mathcal{B};n,P,l,t)\) for any poset \(P\), with equality for chain posets \(C_k\); for \(t=1\), we determine the exact value \(\operatorname{RT}(\mathcal{B}; n, C_k, l) = \operatorname{RT}^{\sharp}(\mathcal{B}; n, C_k, l) = (k-1)(l - 1)\). A universal upper bound for both weak and strong Boolean Ramsey-Turán numbers is given, and \(\operatorname{RT}^\sharp(\mathcal{B}; n, A_k, l, t) = \Theta(n^t)\) for fixed \(k,l,t\) and \(\min\{l-1,k-1\}\ge1\). We also show that for every non-chain poset \(P\), the strong number is \(\Theta(n^t)\) for fixed \(l,t\); moreover, if \(h(P)=r>t\) and \(l(n)=\lfloor M_n^\beta\rfloor\) with \(0<\beta\leq\alpha<1\), then both weak and strong versions satisfy \(\Omega(2^{\beta n}n^{-\beta/2})\) lower bounds. \\[2mm]
\end{abstract}

\section{Introduction}

Turán's theorem \cite{Tu41} states that the balanced complete $s$-partite Turán graph $T_s(n)$ has the maximum edge number among all $K_{s+1}$-free $n$-vertex graphs. These Turán graphs have rigid structures, particularly independent sets of linear size in $n$. A natural question is the maximum edge number of $K_{s+1}$-free $n$-vertex graphs without such rigid structures. First introduced by Sós \cite{ES70} in 1969, such problems form the core of Ramsey-Turán theory.

Let $\alpha(G)$ denote the size of the largest independent set in a graph $G$, and set $c(G)=\alpha(\overline{G})$ (where $\overline{G}$ is the complement of $G$); thus $c(G)$ is the largest $p$ such that $K_p \subseteq G$.

\begin{definition}
The \emph{Ramsey-Turán number} $\operatorname{RT}(n, H, l)$ for a graph $H$ and $n,l\in\mathbb{N}$ is the maximum number of edges in an $n$-vertex $H$-free graph $G$ with $\alpha(G)<l$.
\end{definition}
For $H=K_k$, we abbreviate this to $\operatorname{RT}(n, k, l)$.

\begin{definition}
For an $r$-uniform hypergraph $\mathcal{H}$ and function $f(n)$, the \emph{Ramsey-Turán number} $\operatorname{RT}(n, \mathcal{H}, f(n))$ is the maximum number of edges in an $n$-vertex $r$-uniform $\mathcal{H}$-free hypergraph with independence number at most $f(n)$.
\end{definition}

It was motivated by the classical theorems of Ramsey and Turán and their connections to geometry, analysis, and number theory. Ramsey-Turán theory has attracted a great deal of attention over 60 years; see the nice survey by Simonovits and Sós \cite{SS01} and some papers \cite{LR19,BE76,BLS17,BMMM18,LPR22,BC25}.

Recall that the $n$-dimensional Boolean lattice $B_n$ is the power set of $[n] = \{1,2,\dots,n\}$ equipped with the inclusion relation $\subseteq$, and the \textit{Boolean host poset family} is $\mathcal{B} = \{B_n\}_{n\in\mathbb{N}}$, which satisfies $B_n \subseteq B_{n+1}$ and $|B_n|=2^n < 2^{n+1}=|B_{n+1}|$.

\begin{definition}\label{defi-BRTn_combined}
Let $P$ be a poset, and $n,l,t\in\mathbb{N}$. The \emph{weak (resp. strong) Boolean Ramsey-Turán number for $t$-chains}, denoted $\operatorname{RT}(\mathcal{B}; n, P, l, t)$ (resp. $\operatorname{RT}^{\sharp}(\mathcal{B}; n, P, l, t)$), is the maximum value of $|\mathcal{C}_t(Q)|$ over all posets $Q$ satisfying:
\begin{itemize}
    \item[(1)] $Q \subseteq B_n$;
    \item[(2)] $Q$ is weakly (resp. strongly) $P$-free;
    \item[(3)] $w(Q) < l$,
\end{itemize}
where $\mathcal{C}_t(Q)$ denotes the set of all $t$-chains in $Q$, and $w(Q)$ is the width of the poset $Q$.
\end{definition}

A $t$-chain of a poset is a totally ordered subset of size exactly $t$. In the case $t=1$, 1-chains are singletons and the Boolean Ramsey-Turán number $\operatorname{RT}^{\sharp}(\mathcal{B}; n, P, l)$ abbreviates $\operatorname{RT}^{\sharp}(\mathcal{B}; n, P, l, 1)$.

For nonnegative sequences \(a_n,b_n\), we write \(a_n=O(b_n)\) if
\(a_n\le Cb_n\) for all sufficiently large \(n\) and some constant \(C>0\).
We write \(a_n=\Omega(b_n)\) if \(a_n\ge cb_n\) for all sufficiently large
\(n\) and some constant \(c>0\). We write \(a_n=\Theta(b_n)\) if both
\(a_n=O(b_n)\) and \(a_n=\Omega(b_n)\).

\subsection{Our results}

In this paper we study Ramsey-Tur\'{a}n type extremal functions for posets.

\medskip\noindent\textbf{(1) Weak and strong Boolean Ramsey-Turán numbers for $t$-chains.}
In Section~\ref{sec_2}, we study the weak and strong Boolean Ramsey-Turán numbers for $t$-chains. Theorem \ref{th-relation} establishes the universal inequality $\operatorname{RT}(\mathcal{B};n,P,l,t) \leq \operatorname{RT}^\sharp(\mathcal{B};n,P,l,t)$ for any poset $P$. For the $k$-element chain $C_k$, Theorem \ref{th-chain-equality} shows that weak and strong embeddings are equivalent, hence $\operatorname{RT}(\mathcal{B};n,C_k,l,t) = \operatorname{RT}^\sharp(\mathcal{B};n,C_k,l,t)$. For the core case $t=1$, Theorem~\ref{th-chain-exact-value} gives the exact value
$\operatorname{RT}(\mathcal{B}; n, C_k, l) =\operatorname{RT}^{\sharp}(\mathcal{B}; n, C_k, l) = (k-1)(l - 1)$
whenever $l-1\leq \binom{n-k+2}{\lfloor(n-k+2)/2\rfloor}$; in particular, this holds whenever $n\ge k+l-3$.

\medskip\noindent\textbf{(2) An upper bound for poset Ramsey-Turán numbers.}
In Section~\ref{sec3}, we use Dilworth's theorem to derive a universal upper bound for Boolean Ramsey-Turán numbers. This purely combinatorial bound is valid for all finite posets $P$ and does not use any special forbidden-poset structure.
\begin{theorem}\label{thm:main-universal}
Let $P$ be a finite poset and $l$ a positive integer. For any fixed positive integer $t$, we have 
\[
\operatorname{RT}(\mathcal{B}; n, P, l, t)\le
\operatorname{RT}^\sharp(\mathcal{B}; n, P, l, t) \leq
\begin{cases}
0 & \text{if } l = 1, \\
\displaystyle\sum_{i=1}^{\min(t,l-1)} \binom{l-1}{i} \binom{i(n+1)}{t} & \text{if } l \geq 2.
\end{cases}
\]

As $n \to \infty$ (with fixed $t, l$), the asymptotic upper bound holds:
\[
\operatorname{RT}^\sharp(\mathcal{B}; n, P, l, t) \leq \left( \frac{1}{t!} \sum_{i=1}^{\min(t,l-1)} \binom{l-1}{i} i^t \right) n^t + O(n^{t-1}).
\]
For the special case $t=1$, this reduces to the universal asymptotic bound:
$\operatorname{RT}^\sharp(\mathcal{B}; n, P, l) \leq (l-1) \cdot n + O(1)$.
\end{theorem}

Lemma \ref{lem:antichain_embedding} characterizes strong $A_k$-freeness via width ($w(Q)\leq k-1$). Combining Theorem \ref{thm:main-universal} with a disjoint chain construction, Corollary \ref{cor:antichain-complete} gives matching-order two-sided bounds for $\operatorname{RT}^\sharp(\mathcal{B}; n, A_k, l, t)$, with $\operatorname{RT}^\sharp(\mathcal{B}; n, A_k, l, t)=\Theta(n^t)$ for fixed $k,l,t$ and $L=\min\{l-1,k-1\}\geq1$. We also point out that the weak antichain case is different: weak $A_k$-freeness is equivalent to having fewer than $k$ elements.

\medskip\noindent\textbf{(3) Lower bounds for poset Ramsey-Turán numbers.}
In Section~\ref{sec_4} we prove a general lower bound for non-chain posets:
\begin{theorem}\label{thm:main-lower}
Let $P$ be a finite poset with $w(P)=s\ge 2$. Let $l\ge2$ and put
\[
L=\min\{l-1,s-1\}.
\]
If $n\ge L$ and $1\le t\le n-L+1$, then
\[
\operatorname{RT}^{\sharp}(\mathcal{B};n,P,l,t)\ge
L\binom{n-L+1}{t}.
\]
Consequently, for every fixed non-chain poset $P$ and fixed integers $l,t$ with $l\ge2$, we have
\[
\operatorname{RT}^{\sharp}(\mathcal{B};n,P,l,t)=\Theta(n^t)
\]
whenever $\min\{l-1,w(P)-1\}\ge1$.
\end{theorem}

With the above structural tools and the Chernoff bound for binomial variables, we now restate our main result on the asymptotic lower bound for the Boolean Ramsey-Turán number $\mathrm{RT}(\mathcal{B} ; n, P, l, t)$ for general $t$, which leverages the structural properties from Assumption \ref{ass:boolean}:
\begin{assumption}\label{ass:boolean}
Let $\mathcal{B} = \{B_n\}_{n\in\mathbb{N}}$ denote the sequence of $n$-dimensional Boolean lattices, where $B_n = 2^{[n]}$ is equipped with the inclusion partial order. Let $P$ be a finite poset with height $h(P) = r \geq 2$ (by definition, $P$ contains an $r$-element chain and no longer chains), and let $t$ be an integer with $1 \leq t < r$. Let $M_n = \binom{n}{\lfloor n/2 \rfloor}$ (Sperner number), which satisfies $M_n = \Theta(2^n / \sqrt{n})$ by Stirling's approximation. For fixed $0<\alpha<1$ and all sufficiently large $n$, the following hold:
\begin{itemize}
    \item[(A1)] For any fixed $\gamma > 0$, $M_n^\gamma = \omega(n^c)$ for all fixed $c > 0$.
    \item[(A2)] There exists $\mathcal{L} = \bigcup_{i=1}^{\lfloor M_n^\alpha \rfloor} C_i \subseteq B_n$ where each $C_i$ is an $(r-1)$-element chain, all elements of $C_i$ are incomparable with all elements of $C_j$ for $i \neq j$, and
          \begin{equation*}\label{eq:t-chain-count}
          |\mathcal{C}_t(\mathcal{L})| = \lfloor M_n^\alpha \rfloor \binom{r-1}{t}.
          \end{equation*}
          This construction is feasible for all sufficiently large $n$: since the middle layer of $B_{n-r+2}$ has size $\Theta(M_n)$ and $0<\alpha<1$, it contains at least $\lfloor M_n^\alpha\rfloor$ elements. Take pairwise incomparable distinct subsets $A_i$ from this middle layer (an antichain by Sperner's theorem). Set $S_0=\emptyset$ and $S_j=\{n,n-1,\dots,n-j+1\}$ for $1\le j\le r-2$, and let
          $$C_i=\{A_i\cup S_j:j=0,1,\dots,r-2\}.$$
          For any $i\neq j$, $A_i$ and $A_j$ are incomparable subsets of $[n-r+2]$, so for any $S,T \subseteq \{n-r+3,\dots,n\}$, $A_i \cup S$ and $A_j \cup T$ are incomparable (inclusion would require $A_i \subseteq A_j$ or $A_j \subseteq A_i$, a contradiction).
    \item[(A3)] $\Delta_t = \max_{x\in\mathcal{L}} |\{C\in\mathcal{C}_t(\mathcal{L}) : x\in C\}| = \binom{r-2}{t-1}$ (constant, independent of $n$).
    \item[(A4)] All subsets of $\mathcal{L}$ are weakly $P$-free: every weak embedding maps chains of $P$ to chains in $\mathcal{L}$, while the maximum chain length in $\mathcal{L}$ is $r-1<h(P)$; hence no weak embedding of $P$ into any subset of $\mathcal{L}$ exists.
\end{itemize}
\end{assumption}

\begin{theorem}\label{thm:main}
Under Assumption \ref{ass:boolean}, for any integer-valued function $l(n)$ satisfying $2 \leq l(n) \leq \lfloor M_n^\alpha \rfloor$:
\begin{enumerate}
    \item For fixed $l\geq2$, $\operatorname{RT}(\mathcal{B}; n, P, l, t) \geq \binom{r-1}{t} (l-1)$ for all sufficiently large $n$;
    \item For $l(n)\to\infty$ as $n\to\infty$, $\operatorname{RT}(\mathcal{B}; n, P, l(n), t) \geq (1-o(1)) \binom{r-1}{t} l(n)$ for sufficiently large $n$, where the $o(1)$ term tends to $0$ as $n\to\infty$.
\end{enumerate}
In particular:
\begin{itemize}
    \item For $l(n)$ in the above range with $l(n)=\omega(n^c)$ for any fixed $c>0$, $\operatorname{RT}(\mathcal{B}; n, P, l(n), t) = \omega(n^c)$;
    \item For $l(n)=\lfloor M_n^\beta\rfloor$ with $0<\beta\leq\alpha$, $\operatorname{RT}(\mathcal{B}; n, P, l(n), t) = \Omega(2^{\beta n} n^{-\beta/2})$.
\end{itemize}
Moreover, the same lower bounds hold for the strong Boolean Ramsey-Turán number with the corresponding parameter $l$ or $l(n)$, since every subset constructed in Assumption \ref{ass:boolean}(A2) has height at most $r-1<h(P)$ and hence contains no strong copy of $P$.
\end{theorem}

\section{Preliminary}

Ramsey theory
is a branch of mathematics that studies the conditions of when a combinatorial object necessarily contains some smaller given objects; see \cite{Ramsey30}. 
Ramsey theory has important applications in algebra, geometry, logic, ergodic theory, theoretical computer science, information theory, number theory, and set theory; see the book \cite{GRS90} and survey paper \cite{Rosta04}.

A \emph{partially-ordered set}, or a \emph{poset}, is a pair $(P,\leq )$,
where $P$ is a set and $\leq $ is a relation on $P$
that is reflexive, anti-symmetric, and transitive.
When the relation $\leq$ is clear from the context, we simply write $P$ as a poset.
A pair $x, y\in P$ is \emph{comparable} if $x\leq y$ or $y\leq x$,
and a \emph{$t$-chain} is a set of $t$ distinct pairwise comparable elements. The \emph{$n$-antichain}, denoted $A_n$,
is the poset in which any two elements are not comparable.
Let $(P, \leq_P)$ and $(Q, \leq_Q)$ be posets. We say $Q \subseteq P$ (i.e., $Q$ is a subposet of $P$) if $Q$ is a subset of the underlying set of $P$ (set-theoretic inclusion) and for all $x, y \in Q$, $x \leq_Q y$ if and only if $x \leq_P y$; this corresponds to the special case of a strong embedding via the identity map. For posets $P$ and $Q$, an injection $f : P\to Q$ is a \emph{weak embedding} if $f(x)\leq f(y)$ whenever $x\leq y$, in which case $f(P)$ is called a copy of $P$ in $Q$. An injection $f : P\to Q$ is a \emph{strong embedding} if $f(x)\leq f(y)$ if and only if $x\leq y$, and $f(P)$ is then an \emph{induced copy} of $P$ in $Q$. A poset $Q$ is \emph{weakly $P$-free} if it contains no weak copy of $P$, and \emph{strongly $P$-free} if it contains no induced copy of $P$.

A \emph{Boolean lattice of dimension $n$}, denoted $B_n$, is the
power set of an $n$-element ground set $X$ equipped with inclusion relation.
In this paper,
unless otherwise noted,
we use the set $[n]=\{1,2,...,n\}$
as a ground set of the Boolean lattice $B_n$ of dimension $n$.
The \emph{$k$-th level} of $B_n$ is the collection of all $k$-subsets of $[n]$, denoted by ${[n]\choose k}$.
Let $P$ be a poset.
Let $|P|$ denote the number of elements in $P$.
The \emph{height} $h(P)$ of $P$ is
the maximum number $t$ such that $P$ contains a $t$-chain. The \emph{width} $w(P)$ of $P$ is the number $w$ of elements in a
largest set $\{a_1,..., a_w\}$ where no two elements are comparable.
We write $e(P)$ for the largest $m$
such that $P$ cannot be embedded into any $m$ consecutive layers of $B_{n}$
for any $n$.

By a standard inclusion-exclusion count, the number of $t$-chains in a Boolean lattice has the following form.
\begin{proposition}\label{th-La2}
For positive integers \(p\) and \(t\), the number of \(t\)-chains in \(B_p\) is
\begin{equation*}\label{eq:hpt_def}
h_p(t) = \sum_{i=0}^{t-1} (-1)^{t-1-i} \binom{t-1}{i} (i + 2)^p.
\end{equation*}
In particular, for fixed \(t\), we have $h_p(t)=\Theta((t+1)^p)$ as $p\to\infty$.
\end{proposition}

Let $g(m,N)$ denote the number of strong embeddings of $B_m$ into $B_N$, and let $a(m)$ denote the number of antichains in $B_m$. Axenovich and Walzer \cite{AW17} gave the following estimate.
\begin{theorem}[\cite{AW17}, Count of Strong Embeddings]\label{number_b_n}
Let $m,N$ be two integers with $m \leq N$. Then
$$
\frac{N !}{(N-m) !}(a(m)-m)^{N-m} \leq g(m, N) \leq \frac{N !}{(N-m) !} a(m)^{N-m}.
$$
Here $a(m)=2^{{m\choose \left\lfloor m/2\right\rfloor}(1+O(\log_2 m/m))}$ by the estimate of Kleitman and Markowsky \cite{KM}.
\end{theorem}

In this paper, we consider a $k$-coloring of $t$-chains,
which is an assignment of $k$ colors,
say $1,2, \dots , k$,
to the $t$-chains of a given poset.
If $t=1$, then it is equivalent with a $k$-coloring of elements of the poset.
A colored poset $P$ is \textit{monochromatic} if all of its $t$-chains
share the same color.

Ramsey theory on posets was initiated by Ne\v{s}et\v{r}il and
R\"{o}dl \cite{NesetrilRodl} focusing on induced copies of posets.
%
Kierstead and Trotter \cite{KiersteadTrotter} considered arbitrary hosts
instead of Boolean lattices and
studied Ramsey type problems in terms
the cardinality, height and width.
For more details, we refer to \cite{DKT91,
KiersteadTrotter, McColm, TrotterRamsey}.
Recently, some
researchers \cite{AW17, CS18,
GundersonRodlSidorenko, JohnstonLuMilans, LT22} focused specifically on
finding induced copies of posets in the Boolean lattice.

A related extremal-set-theoretic line of work studies how many $t$-chains a poset-free family can contain.  Katona~\cite{Katona73} determined the maximum number of $2$-chains in a $C_3$-free family, and Gerbner and Patk\'{o}s~\cite{GerbnerPatkos08} extended this to $t$-chains in $C_{t+k}$-free families for all $t,k\ge1$.  This was later generalized to counting copies of one poset $Q$ in families avoiding another poset $P$~\cite{GerbnerKeszeghPatkos20}.  For chain copies, Gerbner, Methuku, Nagy, Patk\'{o}s, and Vizer~\cite{GerbnerMethukuNagyPatkosVizer19} proved a height-based dichotomy: when $h(P)>t$ the extremal order is the same as for forbidding $C_{t+1}$, while when $h(P)\le t$ it is smaller.  Balogh, Martin, Nagy, and Patk\'{o}s~\cite{BaloghMartinNagyPatkos22} further studied height-two forbidden posets for $t=2$.  Unlike these weak, non-induced results, the present paper takes a Ramsey-coloring viewpoint and focuses on strong induced Boolean-lattice copies.

In 1984, Griggs, Stahl, and Trotter \cite{GST84} initiated the study of $t$-chain structures in posets, and their generalized Lubell inequality forms the foundational framework for our analysis of the $t$-chain Lubell function. 
In \cite{CS18}, Cox and Stolee give the following definition of Poset Ramsey number in terms of pographs and weak embeddings, and the host posets they studied are Boolean lattices. Katona et al. \cite{KMOW25} introduced new definitions for posets, considering both strong and weak embeddings.
\begin{definition}\label{Defi-Wposet}
For a family $\mathcal{Q}=\{Q_{n} : n\geq 1\}$ of host posets
such that $Q_n \subseteq Q_{n+1}$ and $|Q_n|<|Q_{n+1}|$ for each $n$,
and given $k$ posets $P_1,P_2,\ldots,P_k$,
the \emph{weak poset Ramsey number
for $t$-chains},
denoted by $\operatorname{R}_{k,t}(\mathcal{Q}\,|\,P_1,P_2,\ldots,P_k)$,
is the smallest number $N$
such that for any $k$-coloring of $t$-chains in $Q_N \in \mathcal{Q}$,
there is a monochromatic copy of $P_i$ in color $i$ for some $1\leq i\leq k$.
\end{definition}

We use the simplified notation $\operatorname{R}_{k,t}(\mathcal{Q}\,|\,P)$ to denote $\operatorname{R}_{k,t}(\mathcal{Q}\,|\,P_1,P_2,\ldots,P_k)$ in the case where $P_1=\cdots=P_k=P$.
When $t=1$, we abbreviate this to $\operatorname{R}_{k}(\mathcal{Q}\,|\,P_1,P_2,\ldots,P_k)$. Moreover, in the case that both $t=1$ and $k=2$, we use the notation $\operatorname{R}(\mathcal{Q}\,|\,P_1,P_2)$.

Similarly, we define the \emph{strong poset Ramsey number} $\operatorname{R}_{k,t}^{\sharp}(\mathcal{Q}\,|\,P_1,P_2,\ldots,P_k)$ for induced copies of $P_i$.
Additionally,
we adopt analogous abbreviations including
$\operatorname{R}_{k,t}^{\sharp}(\mathcal{Q}\,|\,P)$,
$\operatorname{R}_{k}^{\sharp}(\mathcal{Q}\,|\,P_1, P_2, \dots , P_k)$,
and
$\operatorname{R}^{\sharp}(\mathcal{Q}\,|\,P_1, P_2)$. 
Let $\mathcal{B}=\{B_{n} : n\geq 1\}$ denote
the collection of all Boolean lattices.
Then
$\operatorname{R}_{k,t}(\mathcal{B}\,|\,P_1,P_2,\ldots,P_k)$
is called the \emph{Boolean Ramsey number}, see \cite{CS18}.
Similarly,
the variation of strong embeddings has been considered:
in \cite{AW17, LT22},
$\operatorname{R}_{k,t}^{\sharp}(\mathcal{B}\,|\,P_1,P_2,\ldots,P_k)$
is called simply
the \emph{strong Boolean Ramsey number}.

We now present several established findings regarding poset Ramsey numbers.
In the case of a 2-coloring scheme (i.e., $k=2$) applied element-wise (i.e., $t=1$) to an underlying poset,
the Ramsey number $\operatorname{R}^{\sharp}(\mathcal{B}\,|\,B_{n},B_{m})$
corresponding to Boolean lattices $B_n$ and $B_m$
has been investigated by
Axenovich and Walzer \cite{AW17}, Lu and Thompson \cite{LT22},
Gr\'{o}sz, Methuku, and Tompkins \cite{GMT23}, Bohman and Peng \cite{BP23}, and Walzer \cite{Walzer15}.
Axenovich and Winter \cite{AW23, AW24} as well as Winter \cite{Winter23, WinterIII}
explored the characteristics of the poset Ramsey number $\operatorname{R}^{\sharp}(\mathcal{B}\,| P,B_n)$ for a fixed
poset $P$ and an $n$-dimensional Boolean lattice $B_n$, in the case where $n$ tends to be sufficiently large.
Falgas-Ravry et al.~\cite{FMTZ20}
examined Ramsey properties associated with random posets.

Cox and Stolee \cite[Section 3]{CS18}
studied the case of $k$-coloring of $t$-chains in a Boolean lattice
with $k \geq 3$ and $t \geq 2$,
mainly for particular posets with $t = 2$.
Katona et al. \cite{KMOW25} gave several lower and upper bounds on 
the weak and strong poset Ramsey number for $t$-chains.

\section{Relation between the strong and weak cases}\label{sec_2}

A fundamental result connecting antichains and chain decompositions in finite posets is Dilworth's theorem.

\begin{theorem}[Dilworth's Theorem, {\cite{Di50}}]\label{thm:dilworth}
For every finite poset, the size of a largest antichain equals the minimum number of pairwise disjoint chains whose union is the whole poset.
\end{theorem}

We will also use Sperner's theorem for the Boolean lattice.

\begin{theorem}[Sperner's Theorem, {\cite[Theorem 1]{GP19}}]\label{thm:sperner}
If $\mathcal{F}\subseteq 2^{[n]}$ is an antichain, then $|\mathcal{F}|\le \binom{n}{\lfloor n/2\rfloor}$. Moreover, equality holds if and only if $\mathcal{F}=\binom{[n]}{\lfloor n/2\rfloor}$ or $\mathcal{F}=\binom{[n]}{\lceil n/2\rceil}$.
\end{theorem}

To simplify the notation, for a poset $P$ and integers $n,l\in\mathbb{N}$ with $l\ge 1$, we define
\[
S_{\mathrm{weak}}(P;n,l)=\{Q\subseteq B_n:\ Q \text{ is weakly } P\text{-free and } w(Q)<l\}
\]
and
\[
S_{\mathrm{strong}}(P;n,l)=\{Q\subseteq B_n:\ Q \text{ is strongly } P\text{-free and } w(Q)<l\}.
\]
When $n$ and $l$ are clear from the context, we simply write $S_{\mathrm{weak}}(P)$ and $S_{\mathrm{strong}}(P)$.

Using Definition~\ref{defi-BRTn_combined}, we now establish the basic relation between the weak and strong Boolean Ramsey--Tur\'an numbers for $t$-chains. As before, in the special case $t=1$ we abbreviate $\operatorname{RT}(\mathcal{B};n,P,l,1)$ and $\operatorname{RT}^{\sharp}(\mathcal{B};n,P,l,1)$ to $\operatorname{RT}(\mathcal{B};n,P,l)$ and $\operatorname{RT}^{\sharp}(\mathcal{B};n,P,l)$, respectively.

\begin{theorem}\label{th-relation}
For any poset $P$ and integers $n,l,t\in\mathbb{N}$ with $l\ge 1$, we have
\[
\operatorname{RT}(\mathcal{B};n,P,l,t)\le \operatorname{RT}^{\sharp}(\mathcal{B};n,P,l,t).
\]
Moreover, equality holds if and only if there exists a poset $Q_0\subseteq B_n$ such that $Q_0$ is weakly $P$-free, $w(Q_0)<l$, and $|\mathcal{C}_t(Q_0)|=\operatorname{RT}^{\sharp}(\mathcal{B};n,P,l,t)$.
\end{theorem}

\begin{proof}
We first prove that $S_{\mathrm{weak}}(P;n,l)\subseteq S_{\mathrm{strong}}(P;n,l)$. Let $Q\in S_{\mathrm{weak}}(P;n,l)$. Then $Q\subseteq B_n$, $w(Q)<l$, and $Q$ is weakly $P$-free. Suppose, for a contradiction, that $Q\notin S_{\mathrm{strong}}(P;n,l)$. Since the conditions $Q\subseteq B_n$ and $w(Q)<l$ are already satisfied, it follows that $Q$ is not strongly $P$-free. Hence there exists a strong embedding $f:P\to Q$. By definition, a strong embedding is in particular order-preserving, so whenever $x\le_P y$ we have $f(x)\le_Q f(y)$. Therefore $f$ is also a weak embedding of $P$ into $Q$, contradicting the assumption that $Q$ is weakly $P$-free. This proves that every member of $S_{\mathrm{weak}}(P;n,l)$ also belongs to $S_{\mathrm{strong}}(P;n,l)$.

Since $B_n$ is finite, it has only finitely many subposets. Consequently, both $S_{\mathrm{weak}}(P;n,l)$ and $S_{\mathrm{strong}}(P;n,l)$ are finite families of finite posets, and the maxima appearing in the definitions of $\operatorname{RT}(\mathcal{B};n,P,l,t)$ and $\operatorname{RT}^{\sharp}(\mathcal{B};n,P,l,t)$ are well defined. Because $S_{\mathrm{weak}}(P;n,l)\subseteq S_{\mathrm{strong}}(P;n,l)$ and both Ramsey--Tur\'an numbers maximize the same quantity $|\mathcal{C}_t(Q)|$, we obtain
\[
\operatorname{RT}(\mathcal{B};n,P,l,t)
=
\max_{Q\in S_{\mathrm{weak}}(P;n,l)} |\mathcal{C}_t(Q)|
\le
\max_{Q\in S_{\mathrm{strong}}(P;n,l)} |\mathcal{C}_t(Q)|
=
\operatorname{RT}^{\sharp}(\mathcal{B};n,P,l,t).
\]

We next characterize the equality case. Assume first that $\operatorname{RT}(\mathcal{B};n,P,l,t)=\operatorname{RT}^{\sharp}(\mathcal{B};n,P,l,t)$. By the definition of $\operatorname{RT}(\mathcal{B};n,P,l,t)$, there exists some $Q_0\in S_{\mathrm{weak}}(P;n,l)$ such that $|\mathcal{C}_t(Q_0)|=\operatorname{RT}(\mathcal{B};n,P,l,t)$. Since the weak and strong Ramsey--Tur\'an numbers are assumed to be equal, it follows that $|\mathcal{C}_t(Q_0)|=\operatorname{RT}^{\sharp}(\mathcal{B};n,P,l,t)$. Hence $Q_0$ is weakly $P$-free, satisfies $w(Q_0)<l$, and attains the maximum defining the strong Boolean Ramsey--Tur\'an number.

Conversely, suppose that there exists a poset $Q_0\subseteq B_n$ such that $Q_0$ is weakly $P$-free, $w(Q_0)<l$, and $|\mathcal{C}_t(Q_0)|=\operatorname{RT}^{\sharp}(\mathcal{B};n,P,l,t)$. Then $Q_0\in S_{\mathrm{weak}}(P;n,l)$, and therefore the definition of $\operatorname{RT}(\mathcal{B};n,P,l,t)$ gives
\[
\operatorname{RT}(\mathcal{B};n,P,l,t)\ge |\mathcal{C}_t(Q_0)|=\operatorname{RT}^{\sharp}(\mathcal{B};n,P,l,t).
\]
Combining this with the inequality already proved, we conclude that $\operatorname{RT}(\mathcal{B};n,P,l,t)=\operatorname{RT}^{\sharp}(\mathcal{B};n,P,l,t)$.

This completes the proof.
\end{proof}

Let $C_k=(\{a_1,a_2,\dots,a_k\},\le)$ be the $k$-element chain with
$a_1\le a_2\le \cdots \le a_k$, where $k\ge 1$.

We now show that for chain posets the weak and strong Boolean Ramsey--Tur\'an numbers coincide.

\begin{theorem}\label{th-chain-equality}
For any integers $k,n,l,t\in\mathbb{N}$ with $k\ge 1$ and $l\ge 1$, we have
\[
\operatorname{RT}(\mathcal{B};n,C_k,l,t)=\operatorname{RT}^{\sharp}(\mathcal{B};n,C_k,l,t).
\]
\end{theorem}

\begin{proof}
We first show that, for every subposet $Q\subseteq B_n$, a map $f:C_k\to Q$ is a weak embedding if and only if it is a strong embedding.

Indeed, every strong embedding is, by definition, a weak embedding. Conversely, let $f:C_k\to Q$ be a weak embedding. We only need to prove the reverse implication in the order preservation. Assume that $f(a_i)\le_Q f(a_j)$. If $a_i\nleq a_j$, then since $C_k$ is a chain, we must have $a_j\le a_i$. As $f$ is weakly order-preserving, this implies $f(a_j)\le_Q f(a_i)$. Hence both $f(a_i)\le_Q f(a_j)$ and $f(a_j)\le_Q f(a_i)$ hold, so by antisymmetry in $Q$ we get $f(a_i)=f(a_j)$. Since $f$ is injective, this is impossible unless $i=j$. Therefore $a_i\le a_j$. It follows that $f$ is a strong embedding. Thus weak and strong embeddings of $C_k$ are the same.

We next deduce that
$S_{\mathrm{weak}}(C_k;n,l)=S_{\mathrm{strong}}(C_k;n,l)$.
As shown in the proof of Theorem~\ref{th-relation}, one always has
$S_{\mathrm{weak}}(P;n,l)\subseteq S_{\mathrm{strong}}(P;n,l)$ for every poset $P$.
For the reverse inclusion, let $Q\in S_{\mathrm{strong}}(C_k;n,l)$. Then $Q\subseteq B_n$, $w(Q)<l$, and $Q$ is strongly $C_k$-free. If $Q$ were not weakly $C_k$-free, there would exist a weak embedding of $C_k$ into $Q$. By the equivalence proved above, this embedding would in fact be a strong embedding, contradicting the assumption that $Q$ is strongly $C_k$-free. Hence $Q$ is weakly $C_k$-free, and therefore $Q\in S_{\mathrm{weak}}(C_k;n,l)$. This proves the reverse inclusion.

Since the two admissible families are identical and both Ramsey--Tur\'an numbers maximize the same quantity $|\mathcal{C}_t(Q)|$, we conclude that
\[
\operatorname{RT}(\mathcal{B};n,C_k,l,t)=\operatorname{RT}^{\sharp}(\mathcal{B};n,C_k,l,t).
\]
This completes the proof.
\end{proof}

Recall that $\operatorname{RT}^{\sharp}(\mathcal{B}; n, P, l)=\operatorname{RT}^{\sharp}(\mathcal{B}; n, P, l, 1)$. By Theorem~\ref{th-chain-equality}, the weak and strong Boolean Ramsey--Tur\'an numbers coincide for chain posets, so it suffices to consider the strong case. When $t=1$, we have $|\mathcal{C}_1(Q)|=|Q|$, and hence the Ramsey--Tur\'an number is simply the maximum size of an admissible subposet of $B_n$. We now determine this quantity exactly.
\begin{theorem}\label{th-chain-exact-value}
Let $l \in \mathbb{N}$ with $l\geq1$, let $k\geq2$, and let $n\ge k-2$. If
\[
 l-1\leq \binom{n-k+2}{\left\lfloor (n-k+2)/2\right\rfloor},
\]
then
\[
\operatorname{RT}(\mathcal{B}; n, C_k, l) =\operatorname{RT}^{\sharp}(\mathcal{B}; n, C_k, l) = (k-1)(l - 1).
\]
In particular, the same conclusion holds whenever $n\ge k+l-3$.
\end{theorem}

\begin{proof}
If $l=1$, then $w(Q)<1$ forces every admissible poset $Q$ to be empty, and the claimed value is $0$. Hence we may assume $l\ge2$.

We first prove the upper bound. Let $Q\subseteq B_n$ be strongly $C_k$-free with $w(Q)<l$. By Dilworth's theorem, $Q$ can be partitioned into at most $l-1$ pairwise disjoint chains, say $D_1,\dots,D_m$, where $m\le l-1$. Since $Q$ is strongly $C_k$-free, no chain in $Q$ can have size at least $k$, for otherwise such a chain would contain a strong copy of $C_k$. Thus each $D_i$ has size at most $k-1$. Therefore
\[
|Q|=\sum_{i=1}^m |D_i|\le m(k-1)\le (l-1)(k-1).
\]
Since $Q$ was arbitrary, it follows that $\operatorname{RT}^{\sharp}(\mathcal{B}; n, C_k, l)\le (k-1)(l-1)$.

For the lower bound, put $m=n-k+2$. By the assumption on $n$ and $l$, the middle layer of $B_m$ contains at least $l-1$ elements. Choose distinct sets
\[
A_1,A_2,\dots,A_{l-1}\in \binom{[m]}{\lfloor m/2\rfloor}.
\]
These sets are pairwise incomparable. Let $S_0=\emptyset$, and for $1\le j\le k-2$ let
\[
S_j=\{m+1,m+2,\dots,m+j\}\subseteq [n].
\]
For each $i\in[l-1]$, define
\[
C_i=\{A_i\cup S_j:j=0,1,\dots,k-2\}.
\]
Each $C_i$ is a chain of length $k-1$. Moreover, if $i\ne j$, then every element of $C_i$ is incomparable with every element of $C_j$: indeed, an inclusion between $A_i\cup S_a$ and $A_j\cup S_b$ would imply $A_i\subseteq A_j$ or $A_j\subseteq A_i$, contradicting the choice of $A_i$ and $A_j$ from one layer.

Set $Q=\bigcup_{i=1}^{l-1}C_i$. Then $Q$ is the disjoint union of $l-1$ pairwise incomparable chains, each of size $k-1$. Since no chain in $Q$ has size $k$, the poset $Q$ is strongly $C_k$-free. Also, any antichain in $Q$ contains at most one element from each $C_i$, and choosing one element from each $C_i$ gives an antichain of size $l-1$. Hence $w(Q)=l-1<l$. Finally, $|Q|=(l-1)(k-1)$, so
$\operatorname{RT}^{\sharp}(\mathcal{B}; n, C_k, l)\ge (k-1)(l-1)$.
Combining the upper and lower bounds gives the desired strong equality, and Theorem~\ref{th-chain-equality} gives the weak equality.

Finally, if $n\ge k+l-3$, then $m=n-k+2\ge l-1$, and hence
\[
\binom{m}{\lfloor m/2\rfloor}\ge m\ge l-1
\]
for $l\ge2$. Thus the displayed condition is satisfied. This completes the proof.
\end{proof}

\section{Upper bounds for the poset Ramsey-Tur\'an number}\label{sec3}

We begin with a simple monotonicity property of strong embeddings, which will be used repeatedly in the sequel.

\begin{proposition}\label{prop:embedding-monotone}
Let $P$ be a finite poset, and let $f:P\to Q$ be a strong embedding of $P$ into a poset $Q$. Then $w(Q)\ge w(P)$ and $h(Q)\ge h(P)$.
\end{proposition}

\begin{proof}
Let $A\subseteq P$ be an antichain of maximum size, so $|A|=w(P)$. Since $f$ is a strong embedding, incomparable elements of $P$ are mapped to incomparable elements of $Q$. Hence $f(A)$ is an antichain in $Q$ of size $w(P)$, and therefore $w(Q)\ge w(P)$.

Similarly, let $C\subseteq P$ be a chain of maximum size, so $|C|=h(P)$. Since $f$ is order-preserving, the image $f(C)$ is a chain in $Q$ of size $h(P)$. Thus $h(Q)\ge h(P)$.
\end{proof}

We next record a universal chain decomposition for all admissible subposets of the Boolean lattice. This is an immediate consequence of Dilworth's theorem.

\begin{lemma}\label{lem:universal-decomposition}
Let $P$ be a finite poset and let $l$ be a positive integer. If $Q\subseteq B_n$ is a poset with $w(Q)<l$, then there exist pairwise disjoint chains $D_1,D_2,\dots,D_{l-1}$, some of which may be empty, such that
\[
Q=D_1\cup D_2\cup\cdots\cup D_{l-1},
\]
and each $D_i$ satisfies $|D_i|\le n+1$.
\end{lemma}

\begin{proof}
Since $w(Q)<l$ and $w(Q)$ is a nonnegative integer, we have $w(Q)\le l-1$. By Dilworth's theorem, $Q$ can be partitioned into $w(Q)$ pairwise disjoint chains. Adding $(l-1)-w(Q)$ empty chains if necessary, we obtain pairwise disjoint chains $D_1,D_2,\dots,D_{l-1}$ whose union is $Q$.

Finally, each $D_i$ is a chain in $B_n$, and every chain in $B_n$ has size at most $n+1$. Hence $|D_i|\le n+1$ for every $i$.
\end{proof}

Using the universal chain decomposition from Lemma~\ref{lem:universal-decomposition}, we now derive a universal upper bound for the strong Boolean Ramsey--Tur\'an number.

\newtheorem*{mainthm2}{\rm\bf Theorem~\ref{thm:main-universal}}
\begin{mainthm2}[Restated]
Let $P$ be a finite poset and let $l$ be a positive integer. For every fixed positive integer $t$, we have
\[
\operatorname{RT}(\mathcal{B}; n, P, l, t)\le
\operatorname{RT}^{\sharp}(\mathcal{B}; n, P, l, t)\le
\begin{cases}
0, & \text{if } l=1,\\[1ex]
\displaystyle \sum_{i=1}^{\min(t,l-1)} \binom{l-1}{i}\binom{i(n+1)}{t}, & \text{if } l\ge 2.
\end{cases}
\]
Moreover, as $n\to\infty$ with $t$ and $l$ fixed,
\[
\operatorname{RT}^{\sharp}(\mathcal{B}; n, P, l, t)
\le
\left(\frac{1}{t!}\sum_{i=1}^{\min(t,l-1)}\binom{l-1}{i}i^t\right)n^t+O(n^{t-1}).
\]
In particular, when $t=1$, this becomes
\[
\operatorname{RT}^{\sharp}(\mathcal{B}; n, P, l)\le (l-1)n+O(1).
\]
\end{mainthm2}

\begin{proof}
Let $Q\subseteq B_n$ be strongly $P$-free with $w(Q)<l$. If $l=1$, then $w(Q)<1$ forces $Q=\emptyset$, and hence $|\mathcal{C}_t(Q)|=0$. This proves the first case.

Assume now that $l\ge 2$. By Lemma~\ref{lem:universal-decomposition}, there exist pairwise disjoint chains $D_1,D_2,\dots,D_{l-1}$, some of which may be empty, such that $Q=D_1\cup D_2\cup\cdots\cup D_{l-1}$ and $|D_j|\le n+1$ for every $j$.

For each $t$-chain $T\in \mathcal{C}_t(Q)$, let
\[
s(T)=\bigl|\{\,j\in [l-1]: T\cap D_j\neq\emptyset\,\}\bigr|
\]
be the number of chains among $D_1,\dots,D_{l-1}$ that meet $T$. Since $T$ has exactly $t$ elements, we have $1\le s(T)\le \min(t,l-1)$.

For each integer $i$ with $1\le i\le \min(t,l-1)$, define
\[
\mathcal{C}_t^{(i)}(Q)=\{\,T\in \mathcal{C}_t(Q): s(T)=i\,\}.
\]
Then the families $\mathcal{C}_t^{(i)}(Q)$ are pairwise disjoint and their union is $\mathcal{C}_t(Q)$.

Fix such an $i$. To form a $t$-chain $T\in \mathcal{C}_t^{(i)}(Q)$, one must first choose the $i$ chains among $D_1,\dots,D_{l-1}$ that intersect $T$, and there are $\binom{l-1}{i}$ ways to do this. Once these $i$ chains are chosen, the chain $T$ is a $t$-element subset of their union. Since each selected chain has size at most $n+1$, that union has size at most $i(n+1)$. Hence the number of possibilities for $T$ is at most $\binom{i(n+1)}{t}$. Therefore
\[
|\mathcal{C}_t^{(i)}(Q)|\le \binom{l-1}{i}\binom{i(n+1)}{t}.
\]

Summing over all admissible values of $i$, we obtain
\[
|\mathcal{C}_t(Q)|
=\sum_{i=1}^{\min(t,l-1)} |\mathcal{C}_t^{(i)}(Q)|
\le
\sum_{i=1}^{\min(t,l-1)} \binom{l-1}{i}\binom{i(n+1)}{t}.
\]
Since $Q$ was an arbitrary admissible poset in the definition of $\operatorname{RT}^{\sharp}(\mathcal{B}; n, P, l, t)$, this proves the claimed upper bound for the strong number. The corresponding upper bound for the weak number follows from Theorem~\ref{th-relation}.

For the asymptotic estimate, note that for each fixed $i$ and fixed $t$,
\[
\binom{i(n+1)}{t}=\frac{i^t}{t!}n^t+O(n^{t-1})
\qquad (n\to\infty).
\]
Substituting this into the previous bound gives
\[
|\mathcal{C}_t(Q)|
\le
\sum_{i=1}^{\min(t,l-1)} \binom{l-1}{i}
\left(\frac{i^t}{t!}n^t+O(n^{t-1})\right)
=
\left(\frac{1}{t!}\sum_{i=1}^{\min(t,l-1)}\binom{l-1}{i}i^t\right)n^t+O(n^{t-1}),
\]
which yields the asymptotic upper bound for $\operatorname{RT}^{\sharp}(\mathcal{B}; n, P, l, t)$.

Finally, if $t=1$, then only the term $i=1$ appears, and so
\[
\operatorname{RT}^{\sharp}(\mathcal{B}; n, P, l)
\le
\binom{l-1}{1}\binom{n+1}{1}
=
(l-1)(n+1)
=
(l-1)n+O(1).
\]
This completes the proof.
\end{proof}

We next characterize strong embeddings for antichains.

\begin{lemma}\label{lem:antichain_embedding}
Let $A_k$ be the $k$-element antichain, and let $Q\subseteq B_n$ be a poset. Then $Q$ is strongly $A_k$-free if and only if $w(Q)\le k-1$.
\end{lemma}

\begin{proof}
Suppose first that $Q$ is not strongly $A_k$-free. Then there exists a strong embedding $f:A_k\to Q$. Since any two distinct elements of $A_k$ are incomparable, the definition of a strong embedding implies that any two distinct elements of $f(A_k)$ are also incomparable in $Q$. Hence $f(A_k)$ is an antichain of size $k$ in $Q$, and therefore $w(Q)\ge k$.

Conversely, suppose that $w(Q)\ge k$. Then $Q$ contains an antichain $\{s_1,s_2,\dots,s_k\}$ of size $k$. Let $A_k=\{a_1,a_2,\dots,a_k\}$, and define $f:A_k\to Q$ by $f(a_i)=s_i$ for each $i\in [k]$. The map $f$ is injective, and for any $i,j\in [k]$, we have $a_i\le a_j$ if and only if $i=j$, while $f(a_i)\subseteq f(a_j)$ if and only if $i=j$, because $\{s_1,s_2,\dots,s_k\}$ is an antichain. Thus $f$ is a strong embedding of $A_k$ into $Q$, so $Q$ is not strongly $A_k$-free.

Therefore, $Q$ is strongly $A_k$-free if and only if $w(Q)\le k-1$.
\end{proof}

Combining Theorem~\ref{thm:main-universal} with Lemma~\ref{lem:antichain_embedding}, we obtain matching asymptotic bounds for antichains.

\begin{corollary}\label{cor:antichain-complete}
Let $P=A_k$ be the $k$-element antichain, and let $L=\min\{l-1,k-1\}$. Then
\[
\operatorname{RT}^{\sharp}(\mathcal{B}; n, A_k, l, t)\le
\begin{cases}
0, & \text{if } L=0,\\[1ex]
\displaystyle \sum_{i=1}^{\min(t,L)} \binom{L}{i}\binom{i(n+1)}{t}, & \text{if } L\ge 1,
\end{cases}
\]
and, whenever $n\ge L$,
\[
\operatorname{RT}^{\sharp}(\mathcal{B}; n, A_k, l, t)\ge
\begin{cases}
0, & \text{if } L=0 \text{ or } t>n-L+1,\\[1ex]
\displaystyle L\binom{n-L+1}{t}, & \text{if } L\ge 1 \text{ and } t\le n-L+1.
\end{cases}
\]
In particular, for fixed $k,l,t$ and $n\to\infty$, if $L\ge 1$, then $\operatorname{RT}^{\sharp}(\mathcal{B}; n, A_k, l, t)=\Theta(n^t)$, and when $t=1$ we have $\operatorname{RT}^{\sharp}(\mathcal{B}; n, A_k, l)=Ln+O(1)$.
\end{corollary}

\begin{proof}
By Lemma~\ref{lem:antichain_embedding}, a poset $Q\subseteq B_n$ is strongly $A_k$-free if and only if $w(Q)\le k-1$. Together with the condition $w(Q)<l$, this shows that every admissible $Q$ satisfies $w(Q)\le L$.

Applying the same counting argument as in Theorem~\ref{thm:main-universal}, but with $L$ in place of $l-1$, we obtain
\[
\operatorname{RT}^{\sharp}(\mathcal{B}; n, A_k, l, t)\le
\begin{cases}
0, & \text{if } L=0,\\[1ex]
\displaystyle \sum_{i=1}^{\min(t,L)} \binom{L}{i}\binom{i(n+1)}{t}, & \text{if } L\ge 1.
\end{cases}
\]

For the lower bound, assume that $L\ge 1$ and $n\ge L$. For each $i\in [L]$, define
\[
C_i=\bigl\{\{i\}\cup\{L+1,L+2,\dots,L+s\}: s=0,1,\dots,n-L\bigr\}\subseteq B_n,
\]
and let $Q=\bigcup_{i=1}^L C_i$. Each $C_i$ is a chain of size $n-L+1$, and if $X\in C_i$ and $Y\in C_j$ with $i\ne j$, then $i\in X\setminus Y$ and $j\in Y\setminus X$, so $X$ and $Y$ are incomparable. Hence elements from different chains are incomparable, and thus every antichain in $Q$ contains at most one element from each $C_i$. Conversely, choosing one element from each chain gives an antichain of size $L$, so $w(Q)=L\le k-1$. By Lemma~\ref{lem:antichain_embedding}, $Q$ is strongly $A_k$-free, and since also $L\le l-1$, we have $w(Q)<l$. Therefore $Q$ is admissible in the sense of Definition~\ref{defi-BRTn_combined}.

Moreover, every $t$-chain in $Q$ lies entirely inside one of the chains $C_i$, while each $C_i$ contains exactly $\binom{n-L+1}{t}$ such $t$-chains. It follows that $|\mathcal{C}_t(Q)|=L\binom{n-L+1}{t}$, which proves the lower bound.

For fixed $k,l,t$ and $L\ge 1$, the lower bound satisfies $L\binom{n-L+1}{t}=\frac{L}{t!}n^t+O(n^{t-1})=\Omega(n^t)$, whereas the upper bound is clearly $O(n^t)$. Hence $\operatorname{RT}^{\sharp}(\mathcal{B}; n, A_k, l, t)=\Theta(n^t)$. When $t=1$, the upper bound is at most $L(n+1)$ and the lower bound is at least $L(n-L+1)$; their difference is the constant $L^2$, so $\operatorname{RT}^{\sharp}(\mathcal{B}; n, A_k, l)=Ln+O(1)$.
\end{proof}

\begin{remark}
The disjoint-chain construction used in Corollary~\ref{cor:antichain-complete} is asymptotically optimal in order for every fixed $t$, and it is tight up to an additive constant when $t=1$. For $t\ge 2$, the gap between the upper and lower bounds is caused by the universal upper bound allowing potential cross-chain contributions, whereas in our construction every $t$-chain lies inside a single chain.
\end{remark}

\begin{remark}\label{rem:weak-antichain}
The weak antichain case is quite different from the strong one. Since $A_k$ has no nontrivial order relations, every $k$ distinct elements of an arbitrary poset form a weak copy of $A_k$. Hence a poset $Q$ is weakly $A_k$-free if and only if $|Q|\le k-1$. Consequently, if $l\ge2$ and $n\ge k-2$, then
\[
\operatorname{RT}(\mathcal{B};n,A_k,l,t)=\binom{k-1}{t},
\]
with the usual convention that $\binom{k-1}{t}=0$ when $t>k-1$. For $l=1$, the value is $0$.
\end{remark}

\section{Structured lower bounds for Boolean Ramsey-Tur\'an numbers}\label{sec_4}

We now prove the lower-bound results stated in the introduction. The first one is a direct disjoint-chain construction for all non-chain posets.

\newtheorem*{mainthm3}{\rm\bf Theorem~\ref{thm:main-lower}}
\begin{mainthm3}[Restated]
Let $P$ be a finite poset with $w(P)=s\ge 2$. Let $l\ge2$ and put
$L=\min\{l-1,s-1\}$.
If $n\ge L$ and $1\le t\le n-L+1$, then
\[
\operatorname{RT}^{\sharp}(\mathcal{B};n,P,l,t)\ge
L\binom{n-L+1}{t}.
\]
Consequently, for every fixed non-chain poset $P$ and fixed integers $l,t$ with $l\ge2$, we have
\[
\operatorname{RT}^{\sharp}(\mathcal{B};n,P,l,t)=\Theta(n^t)
\]
whenever $\min\{l-1,w(P)-1\}\ge1$.
\end{mainthm3}

\begin{proof}
For each $i\in [L]$, define
\[
C_i=\bigl\{\{i\}\cup\{L+1,L+2,\dots,L+j\}:j=0,1,\dots,n-L\bigr\}\subseteq B_n,
\]
and set $Q=\bigcup_{i=1}^{L}C_i$.

Each $C_i$ is a chain of size $n-L+1$. If $X\in C_i$ and $Y\in C_j$ with $i\ne j$, then $i\in X\setminus Y$ and $j\in Y\setminus X$, so $X$ and $Y$ are incomparable. Hence $Q$ is the disjoint union of $L$ pairwise incomparable chains. It follows that $w(Q)=L<l$.

Since $L\le s-1<w(P)$, Proposition~\ref{prop:embedding-monotone} implies that $Q$ contains no strong copy of $P$; otherwise such a copy would have width at least $w(P)=s$, whereas $Q$ has width $L\le s-1$. Thus $Q$ is strongly $P$-free and is admissible in the definition of $\operatorname{RT}^{\sharp}(\mathcal{B};n,P,l,t)$.

Every $t$-chain of $Q$ lies entirely inside one of the chains $C_i$, because elements from different chains are incomparable. Each chain $C_i$ contains exactly $\binom{n-L+1}{t}$ $t$-chains. Therefore
\[
|\mathcal{C}_t(Q)|=L\binom{n-L+1}{t},
\]
which proves the claimed lower bound.

For fixed $P,l,t$, the quantity $L$ is a positive constant, and hence
\[
L\binom{n-L+1}{t}=\frac{L}{t!}n^t+O(n^{t-1})=\Omega(n^t).
\]
The universal upper bound in Theorem~\ref{thm:main-universal} gives $O(n^t)$. Therefore $\operatorname{RT}^{\sharp}(\mathcal{B};n,P,l,t)=\Theta(n^t)$. This completes the proof.
\end{proof}

For bounding the probability that the number of selected parent chains deviates from its expectation, we use the following Chernoff bound.

\begin{theorem}[Chernoff Bound, \cite{MR95}]\label{thm:chernoff}
If $X$ is a binomial random variable with expectation $\mu$, and $\epsilon>0$, then $$\Pr[X \le (1-\epsilon)\mu]\le \exp\left(-\frac{\epsilon^2\mu}{2}\right),$$ and $$\Pr[X \ge (1+\epsilon)\mu]\le \left(\frac{e^\epsilon}{(1+\epsilon)^{1+\epsilon}}\right)^{\mu}.$$
\end{theorem}

Building on the structural construction of Boolean lattices in Assumption \ref{ass:boolean}, we restate our core result on the asymptotic lower bound for the Boolean Ramsey-Tur\'an number $\mathrm{RT}(\mathcal{B}; n, P, l, t)$ as follows.

\newtheorem*{mainthm4}{\rm\bf Theorem~\ref{thm:main}}
\begin{mainthm4}[Restated]
Under Assumption \ref{ass:boolean}, for any integer-valued function $l(n)$ satisfying $2 \leq l(n) \leq \lfloor M_n^\alpha \rfloor$:
\begin{enumerate}
    \item[(i)] For fixed $l\geq 2$ and all sufficiently large $n$, 
    \begin{equation}\label{eq:rt_fixed_l}
    \operatorname{RT}(\mathcal{B}; n, P, l, t) \geq \binom{r-1}{t}(l-1).
    \end{equation}
    \item[(ii)] For $l(n)\to\infty$ as $n\to\infty$, 
    \begin{equation}\label{eq:rt_large_l}
    \operatorname{RT}(\mathcal{B}; n, P, l(n), t) \geq (1-o(1))\binom{r-1}{t}\,l(n).
    \end{equation}
\end{enumerate}
In particular:
\begin{itemize}
    \item For $l(n)$ in the above range with $l(n)=\omega(n^c)$ for any fixed $c>0$, $\operatorname{RT}(\mathcal{B}; n, P, l(n), t)=\omega(n^c)$.
    \item For $l(n)=\lfloor M_n^\beta\rfloor$ with $0<\beta\leq\alpha$, 
    \begin{equation}\label{eq:rt_beta}
    \operatorname{RT}(\mathcal{B}; n, P, l(n), t)=\Omega(2^{\beta n}n^{-\beta/2}).
    \end{equation}
\end{itemize}
Moreover, the same lower bounds hold for the strong Boolean Ramsey-Tur\'an number with the corresponding parameter $l$ or $l(n)$.
\end{mainthm4}
\begin{proof}

By Assumption \ref{ass:boolean}(A4), every subset of $\mathcal{L}$ is weakly $P$-free. Since $\mathcal{L}\subseteq B_n$, any subset $Q\subseteq \mathcal{L}$ with $w(Q)<l(n)$ is admissible in the definition of $\operatorname{RT}(\mathcal{B}; n, P, l(n), t)$. Thus it suffices to construct suitable subsets of $\mathcal{L}$ and count their $t$-chains.

Write 
\begin{equation}\label{eq:def_M_ct}
M=\lfloor M_n^\alpha\rfloor \quad \text{and} \quad c_t=\binom{r-1}{t}.
\end{equation}
By Assumption \ref{ass:boolean}(A2), the family $\mathcal{L}$ is the disjoint union of $M$ chains, namely $\mathcal{L}=\bigcup_{i=1}^{M} C_i$, where each $C_i$ has exactly $r-1$ elements, the chains are pairwise incomparable, and each $C_i$ contains exactly $c_t=\binom{r-1}{t}$ distinct $t$-chains (consistent with \eqref{eq:def_M_ct}).

We consider two cases.

\setcounter{case}{0}

\begin{case}
Fixed $l\geq 2$.
\end{case}

For all sufficiently large $n$, we have $M\ge l-1$. Choose any $l-1$ distinct chains $C_{i_1},\dots,C_{i_{l-1}}$ from $\mathcal{L}$, and set $Q=\bigcup_{j=1}^{l-1} C_{i_j}$. Since different parent chains are pairwise incomparable by Assumption \ref{ass:boolean}(A2), any antichain in $Q$ contains at most one element from each chosen chain. On the other hand, choosing one element from each of the $l-1$ chains gives an antichain of size $l-1$. Hence $w(Q)=l-1<l$. By Assumption \ref{ass:boolean}(A4), the poset $Q$ is weakly $P$-free.

Moreover, every $t$-chain of $Q$ lies entirely inside one of the chosen chains, and each chosen chain contributes exactly $c_t=\binom{r-1}{t}$ $t$-chains. Therefore $|\mathcal{C}_t(Q)|=(l-1)\binom{r-1}{t}$. Since $Q$ is admissible, it follows that \eqref{eq:rt_fixed_l} holds.

\begin{case}
$l(n)\to\infty$ as $n\to\infty$.
\end{case}

We now use a random selection of whole parent chains. Let 
$\delta_n=l(n)^{-1/3}$.
Then $\delta_n\to 0$ and $\delta_n^2 l(n)=l(n)^{1/3}\to\infty$ as $n\to\infty$. Since $l(n)\le M$ (from \eqref{eq:def_M_ct}), for all sufficiently large $n$ we may define 
\begin{equation*}\label{eq:def_pn}
p_n=\frac{(1-\delta_n)l(n)}{M}\in (0,1].
\end{equation*}

Construct a random subset $\widetilde{Q}\subseteq \mathcal{L}$ by independently including each whole chain $C_i$ with probability $p_n$, and excluding it otherwise. Let $N$ denote the number of selected parent chains. Then $N\sim \operatorname{Bin}(M,p_n)$ and 
$\mathbb{E}[N]=Mp_n=(1-\delta_n)l(n)$.

Because distinct parent chains are pairwise incomparable, every $t$-chain of $\widetilde{Q}$ lies entirely inside a single selected parent chain. Since each selected parent chain contributes exactly $c_t=\binom{r-1}{t}$ $t$-chains, we have $|\mathcal{C}_t(\widetilde{Q})|=c_t\,N$. Also, any antichain in $\widetilde{Q}$ contains at most one element from each selected parent chain, while choosing one element from each selected chain gives an antichain of size $N$. Hence $w(\widetilde{Q})=N$.

We claim that with positive probability, $(1-2\delta_n)l(n)\le N<l(n)$. Indeed, since $\mathbb{E}[N]=(1-\delta_n)l(n)$, the upper-tail event $N\ge l(n)$ can be written as $$N\ge \left(1+\frac{\delta_n}{1-\delta_n}\right)\mathbb{E}[N].$$ Applying the Chernoff bound (Theorem \ref{thm:chernoff}) with $\epsilon_n=\frac{\delta_n}{1-\delta_n}>0$, we obtain
\begin{equation}\label{eq:chernoff_upper}
\mathbb{P}(N\ge l(n))\le \left(\frac{e^{\epsilon_n}}{(1+\epsilon_n)^{1+\epsilon_n}}\right)^{\mathbb{E}[N]}.
\end{equation}
Since $\epsilon_n\to 0$, there exists an absolute constant $c>0$ such that for all sufficiently large $n$, 
$$
\left(\frac{e^{\epsilon_n}}{(1+\epsilon_n)^{1+\epsilon_n}}\right)\le \exp(-c\epsilon_n^2),
$$ 
and therefore \eqref{eq:chernoff_upper} implies $$\mathbb{P}(N\ge l(n))\le \exp\!\big(-c\epsilon_n^2\mathbb{E}[N]\big).$$ Because $\epsilon_n=\Theta(\delta_n)$ and $\mathbb{E}[N]=\Theta(l(n))$, it follows that $$\epsilon_n^2\mathbb{E}[N]=\Theta(\delta_n^2 l(n))=l(n)^{1/3}\to\infty,$$ and hence $\mathbb{P}(N\ge l(n))=o(1)$.

Similarly, the lower-tail event $N<(1-2\delta_n)l(n)$ can be rewritten as $$N\le \left(1-\frac{\delta_n}{1-\delta_n}\right)\mathbb{E}[N].$$ Applying the lower-tail form of the Chernoff bound with the same parameter $\epsilon_n=\delta_n/(1-\delta_n)$, we get
\begin{equation}\label{eq:chernoff_lower}
\mathbb{P}\bigl(N<(1-2\delta_n)l(n)\bigr)\le \exp\!\left(-\frac{\epsilon_n^2\mathbb{E}[N]}{2}\right).
\end{equation}
By the same reasoning as above, the right-hand side of \eqref{eq:chernoff_lower} is $o(1)$.

Therefore, we have 
$$\mathbb{P}\Big((1-2\delta_n)l(n)\le N<l(n)\Big)\ge 1-\mathbb{P}\bigl(N<(1-2\delta_n)l(n)\bigr)-\mathbb{P}(N\ge l(n))=1-o(1)>0$$ for all sufficiently large $n$. Hence there exists a deterministic subset $Q^*\subseteq \mathcal{L}$ such that $(1-2\delta_n)l(n)\le w(Q^*)<l(n)$ 
and 
$$|\mathcal{C}_t(Q^*)|=c_t\,N\ge (1-2\delta_n)c_t\,l(n)=(1-2\delta_n)\binom{r-1}{t}l(n).$$

By Assumption \ref{ass:boolean}(A4), the subset $Q^*$ is weakly $P$-free. Therefore $Q^*$ is admissible in the definition of $\operatorname{RT}(\mathcal{B}; n, P, l(n), t)$, and so $\operatorname{RT}(\mathcal{B}; n, P, l(n), t)\ge (1-2\delta_n)\binom{r-1}{t}l(n)$. Since $\delta_n=l(n)^{-1/3}\to 0$, we conclude that \eqref{eq:rt_large_l} holds.
This proves parts (i) and (ii).

For the first consequence, if $l(n)$ is in the above range and $l(n)=\omega(n^c)$ for any fixed $c>0$, then by \eqref{eq:rt_large_l}, $(1-o(1))\binom{r-1}{t}l(n)=\omega(n^c)$, and hence $\operatorname{RT}(\mathcal{B}; n, P, l(n), t)=\omega(n^c)$.

For the second consequence, if $l(n)=\lfloor M_n^\beta\rfloor$ with $0<\beta\le \alpha$, then by Stirling's approximation,
\begin{equation*}\label{eq:stirling_Mn}
M_n=\binom{n}{\lfloor n/2\rfloor}=\Theta\!\left(\frac{2^n}{\sqrt{n}}\right).
\end{equation*}
Therefore, we have $$l(n)=\lfloor M_n^\beta\rfloor=\Theta\!\left(\left(\frac{2^n}{\sqrt{n}}\right)^\beta\right)=\Theta\!\left(2^{\beta n}n^{-\beta/2}\right),$$ and combining with \eqref{eq:rt_large_l} we conclude that \eqref{eq:rt_beta} holds.

Finally, the same lower bounds hold for the corresponding strong Boolean Ramsey-Turán number. Indeed, every subset of $\mathcal{L}$ has height at most $r-1<h(P)$, so it contains neither a weak copy nor a strong copy of $P$. Hence all subsets constructed above are also admissible for the strong Boolean Ramsey-Tur\'an number. This completes the proof.
\end{proof}

Having established the asymptotic lower bound above, we next show that this construction is asymptotically tight up to a constant factor inside the structured family $\mathcal{L}$.

\begin{corollary}\label{cor:tightness}
Under Assumption \ref{ass:boolean}, for any $Q\subseteq\mathcal{L}$ with $w(Q)<l(n)$, we have
\[
|\mathcal{C}_t(Q)| \leq \binom{r-1}{t}(l(n)-1).
\]
\end{corollary}

\begin{proof}
Let $S=\{\,i : C_i\cap Q\neq\emptyset\,\}$ be the set of parent chains with non-empty intersection with $Q$.

By Assumption \ref{ass:boolean}(A2), elements from different parent chains are pairwise incomparable. Hence any antichain in $Q$ contains at most one element from each parent chain in $S$, so $w(Q)\le |S|$. Conversely, for each $i\in S$, choose one element of $C_i\cap Q$. Since these chosen elements lie in distinct parent chains, they are pairwise incomparable, and therefore form an antichain of size $|S|$. Thus $w(Q)\ge |S|$. Therefore $w(Q)=|S|$.
Since $w(Q)<l(n)$, it follows that $|S|=w(Q)\le l(n)-1$.

Now every $t$-chain of $Q$ lies inside a single parent chain, because different parent chains are pairwise incomparable. Also, each parent chain $C_i$ has exactly $r-1$ elements, so any subset of $C_i$ contains at most $\binom{r-1}{t}$ $t$-chains. Hence
\[
|\mathcal{C}_t(Q)|=\sum_{i\in S} |\mathcal{C}_t(Q\cap C_i)|\le \sum_{i\in S}\binom{r-1}{t}=\binom{r-1}{t}|S|\le \binom{r-1}{t}(l(n)-1).
\]
This completes the proof.
\end{proof}

\end{document}